\documentclass[12pt]{amsart}
\usepackage{amssymb,amsfonts,amsmath,amsopn,amstext,amscd,color,latexsym,mathrsfs}
\usepackage[shortlabels]{enumitem}
\usepackage{mathabx}
\usepackage[utf8]{inputenc}
\usepackage[all, cmtip]{xy}
\usepackage{mathtools}
\usepackage{enumitem}

\usepackage[pdftex, colorlinks=true, linktoc=all]{hyperref}

\theoremstyle{plain}

\newtheorem{thm}{\bf Theorem}[section]
\newtheorem{lem}[thm]{\bf Lemma}
\newtheorem{cor}[thm]{\bf Corollary}
\newtheorem{prop}[thm]{\bf Proposition}

\theoremstyle{definition}
\newtheorem{nota}[thm]{\bf Notations}

\newtheorem{rem}[thm]{\bf Remark}
\newtheorem{example}[thm]{\bf Example}
\newtheorem{claim}[thm]{\bf Claim}
\newtheorem{question}[thm]{\bf Question}

\theoremstyle{definition}

\newtheorem{defn}[thm]{\bf Definition}

\usepackage{chngcntr}

\counterwithin*{thm}{section}

\newcommand{\bbA}{{\mathbb A}}

\newcommand{\bbC}{{\mathbb C}}

\newcommand{\bbP}{{\mathbb P}}

\newcommand{\bbZ}{{\mathbb Z}}

\newcommand{\bG}{{\bf G}}

\newcommand{\cA}{{\mathcal A}}

\newcommand{\cC}{{\mathcal C}}

\newcommand{\cE}{{\mathcal E}}
\newcommand{\cF}{{\mathcal F}}

\newcommand{\cK}{{\mathcal K}}

\newcommand{\cP}{{\mathcal P}}

\newcommand{\cR}{{\mathcal R}}

\newcommand{\cpH}{{^p\mathcal{H}}}

\newcommand{\rD}{{\rm D}}

\newcommand{\rK}{{\rm K}}

\newcommand{\rT}{{\rm T}}

\newcommand{\sE}{{\mathscr E}}
\newcommand{\sF}{{\mathscr F}}

\newcommand{\sK}{{\mathscr K}}
\newcommand{\sL}{{\mathscr L}}
\newcommand{\sM}{{\mathscr M}}

\newcommand{\sQ}{{\mathscr Q}}

\newcommand{\sS}{{\mathscr S}}

\newcommand{\tp}{\tilde{p}}

\newcommand{\Ext}{{\rm Ext}}

\newcommand{\Hom}{{\rm Hom}}

\newcommand{\Spec}{{\rm Spec}}

\newcommand{\isom}{\cong}

\begin{document}
\title{Brylinski-Radon transformation in characteristic $p>0$}
\author{Deepam Patel}
\address{Department of Mathematics, Purdue University,
150 N. University Street, West Lafayette, IN 47907, U.S.A.}
\email{patel471@purdue.edu}
\author{K.V. Shuddhodan}
\address{Institut des Hautes \'Etudes Scientifiques, Universit\'e Paris-Scalay, CNRS, Laboratoire Alexandre Grothendieck, Le Bois-Marie 35 rte de Chartres, 91440 Bures-sur-Yvette, France}
\email{kvshud@ihes.fr}

\thanks{KVS was supported by the CARMIN project fellowship.}

\begin{abstract}
In this article, we characterize the image of the Brylinski-Radon transform in characteristic $p>0$ via Beilinson's theory of singular supports. We also provide an alternate proof of Brylinski's results over $\bbC$, which also works for sheaves with finite coefficients. Along the way, we also obtain a microlocal criterion for the descent of perverse sheaves which could be of independent interest.
\end{abstract}
\maketitle

\tableofcontents

\section{Introduction}\label{introduction}
In \cite{Br}, Brylinski introduced topological (and geometric) versions of the classical Radon transforms and proved some fundamental properties for these transforms.~The theory has had numerous applications including to Lefschetz theory \cite[I.III]{Br}, \cite[Chapter IV]{KW}. More recently the Radon transform was crucially used by Beilinson \cite{B} and Saito \cite{Sai} respectively in the construction of singular support and characteristic cycle for constructible sheaves in the algebraic setting over arbitrary perfect fields. The main result of this article is to use the theory of singular supports to characterize the image of the Radon transform, generalizing the work of Brylinski to arbitrary base fields (and finite coefficents). In particular, we answer a question raised by Brylinski \cite[Section 5.11]{Br}.\\

\subsection{Summary of results:}\label{localization_category}
\subsubsection{Singular supports of \'etale sheaves:}Let $k$ be an algebraically closed field of char.~$p \geq 0$, $\ell \neq p$ a fixed prime, $\Lambda = \bbZ/l^n\bbZ$, and $\rD^{b}_{ctf}(X,\Lambda)$ denote the derived category of bounded constructible \'etale sheaves of $\Lambda$-modules with finite tor-dimension on $X$. In the rest of the article, we denote this category simply by $\rD^b_c(X)$. Given $K \in \rD^b_c(X)$, $K(n)$ will denote the usual Tate twist of $K$. If $X$ is smooth and $K \in \rD^{b}_{c}(X)$, then Beilinson \cite{B} defined the singular support $SS(K) \subset \rT^*X$ (see \ref{RecSS} for a brief summary about singular supports). This is a closed conical subset of $\rT^*X$, and for $K \neq 0$, $SS(K)$ is equidimensional of dimension equal to $\dim(X)$. Moreover when $\text{char}(k)=0$, $SS(K)$ is Lagrangian\footnote{A closed conical subset of $\rT^*X$ is said to be \textit{Lagrangian} if the smooth locus of the closed subset is both isotropic and involutive with respect to the natural symplectic structure on $\rT^*X$.} \cite[Proposition 2.2.7]{Sai}. However this fails in positive characteristic \cite[Example 1.3]{B}. \\

\subsubsection{Main Result:} Let $\cP(X) \subset \rD^{b}_c(X)$ denote the abelian category of perverse sheaves (w.r.t middle perversity). In the following, given an object $\rK \in \rD^{b}_c(X)$, let $\cpH^{i}(K) \in \cP(X)$ denote the $i$-th perverse cohomology sheaf. If $X$ is smooth\footnote{In this article, varieties smooth over $k$ shall be assumed to be connected.} over $k$ of dimension $n$, let $\cC(X) \subset \cP(X)$ denote the full Serre subcategory of locally constant perverse sheaves (i.e.~complexes of the form $\sL[n]$ where $\sL$ is a locally constant constructible sheaf on $X$), and $\cA(X)$ the corresponding quotient category. One can realize $\cA(X)$ as the heart of the induced perverse t-structure on a localized triangulated category $\rD^{b}_c(X)_T$ obtained by localizing $\rD^{b}_c(X)$ along the multiplicative set of morphisms $f$ such that $\text{ker}(\cpH^{i}(f))$ and $\text{coker}(\cpH^{i}(f))$ are locally constant perverse sheaves for all $i$. As above, let $\cpH_T^{i}(K) \in \cA(X)$ denote the $i$-th `perverse cohomology' sheaf of the image of $K$ in $\rD^{b}_c(X)_T$. \\

We now recall the Brylinski-Radon transform. Let $\mathbb{P}$ denote projective space of dimension $n \geq 2$ over $k$, and $Y := \bG(d)$ denote the denote the Grassmanian of $d$-planes (where $1 \leq d \leq n-1$) in $\mathbb{P}$. Consider the incidence correspondence $Q \subset \mathbb{P} \times Y$. The Brylinski-Radon transform is defined as follows. Consider the diagram:
$$
\xymatrix{
 & Q \ar[dr]^{p_2} \ar[dl]_{p_1} &\\
 \bbP &  & Y, }
 $$
 where $p_i$ are the natural projections.
Given $\cK \in \rD^{b}_c(\bbP^n)$, let $\mathcal{R}(\cK):= p_{2,*}p_1^{\dagger}\cK \in \rD^{b}_c(Y)$\footnote{Give a smooth morphism $f:X \to S$ of relative dimension $d$ with geometrically connected fibers, we set $f^{\dagger}:=f^*[d]$. Also, by $f^{\dagger}\cP(S)$, we mean the full subcategory $\cP(X)$ consisting of perverse sheaves of the form $f^{\dagger}M$ \cite[Proposition 4.2.5]{BBDG18}, where $M$ is a perverse sheaf on $S$.}. Similarly, we set $\check{\cR}(\cK) := p_{1,*}p_2^{\dagger}(\cK)$. \\
%

Let $C \subset \rT^*\bbP$ be a closed conical subset. The \textit{Brylinski-Radon} transform of $C$ is defined to be $p_{2\circ}p_1^{\circ}(C)$ (see Section \ref{RecSS} for the notation). A closed conical subset of $\rT^*Y$ is said to be in the \textit{image of the Brylinski-Radon} transform if it is contained in the Brylinski-Radon transform of a closed conical subset of $\rT^*\bbP$.\\

 It follows from \cite[Theorem 1.4, (ii)]{B} that perverse sheaves whose singular support is in the image of the Brylinski-Radon transform form an abelian subcategory of $\cP(Y)$ (denoted below by $\cP(Y)_{\text{Rad}}$) which is stable under extensions. Let $\cA(Y)_{\text{Rad}}$ be the full abelian subcategory of $\cA(Y)$, consisting of objects which are images of $\cK \in \cP(Y)_{\text{Rad}}$. It is easy to see that both $\cR$ and $\check{\cR}$ naturally induce functors between $\rD^b_c(\bbP)_T$ and $\rD^b_c(Y)_T$. We are now ready to state the main result of this article.

\begin{thm}\label{thm:mainthm}
With notation as above:
\begin{enumerate}
\item $\cR$ is $t$-exact for the perverse $t$-structures on $\rD^b_c(\bbP)_T$ and $\rD^b_c(Y)_T$.

\item The functor $\cpH_T^{d(n-d-1)} \circ \check{\cR}(d(n-d)) \circ \cpH_T^{0} \circ \cR$ is naturally equivalent to the identity functor on $\cA(\bbP)$.

\item The functor $\cpH^{0}_T \circ \cR$ induces an equivalence of categories between $\cA(\bbP)$ and $\cA(Y)_{\text{Rad}}$.

\end{enumerate}

 \end{thm}

\subsection{Comparison with previous work}\label{comparison_previous}
\begin{enumerate}

\item If $k = \bbC$, and one considers constructible sheaves in the classical topology with $\bbC$ coefficients, then this is one of the main results of Brylinski \cite[Th\'eor\`eme 5.5]{Br}. The problem of a characteristic $p$ analogue of Brylinski's theorem was already posed as a question by Brylinski \cite[5.11]{Br}. The results of this article answer this question in the affirmative albeit, with appropriate modifications to account for wild ramification.

\item If $\text{char}(k)=0$, \cite[Proposition 2.2.7]{Sai} and \cite[Lemme 5.6, (d)]{Br} imply that one can alternatively describe $\cP(Y)_{\text{Rad}}$ as those perverse sheaves who singular support is contained in $p_{2\circ}p_1^{\circ}\rT^*\bbP$. In particular, the statement of Theorem \ref{thm:mainthm} is consistent with the analogous statement proved by Brylinski over the complex numbers.

\item If $d = n-1$, then the aforementioned theorem gives an equivalence of categories between $\cA(\bbP)$ and $\cA(\check{\bbP})$. In this setting, Brylinski \cite[Corollaire 9.15]{Br} also proves the result over an algebraic closure of a finite field as an application of the Deligne-Fourier transform in characteristic $p>0$.
 
\end{enumerate}

\subsection{Idea of the proof}
In this section we briefly describe the ideas underlying the proof of Theorem \ref{thm:mainthm}. 

\subsubsection{Proof of Theorem \ref{thm:mainthm}, (1):}

The proof is an easy application of Artin vanishing and is along the lines of the proof in \cite[Chapter IV, Corollary 2.3]{KW}, where the case of $n=d-1$ is handled. The proof in \cite[Th\'eor\`eme 5.5, (1)]{Br} is in comparison  microlocal in nature and does not carry over when the coefficients are finite. 

\subsubsection{Proof of Theorem \ref{thm:mainthm}, (2):}

The essential point here is to understand the pushforward of the constant sheaf along the map $Q \times_Y Q \to \bbP \times \bbP$. This map is smooth outside the diagonal $\bbP$, however the fibers of the map are Grassmannians. This allows us to compute $\cR^{\vee} \circ \cR$ in the localized category $\rD^{b}_c(\bbP)_T$ (see (\ref{M5})) and deduce Theorem \ref{thm:mainthm}, (2). We do so without recourse to the decomposition theorem which is technically important for us since we allows finite coefficients. As a corollary of the proof we are able to show that $\cpH^{0} \circ \cR$ is fully faithful and induces isomorphism on $\text{Ext}^1$ (see Corollary \ref{fully_faithfullness}).

\subsubsection{Proof of Theorem \ref{thm:mainthm}, (3):}

The proof of Theorem \ref{thm:mainthm}, (3) constitutes the technical heart of the paper. The first step is to prove a microlocal criterion for the descent of perverse sheaves. More precisely we prove the following statement which generalizes a result of Laumon\footnote{We thank Ahmed Abbes for pointing out the connection of our result with Laumon's work.} \cite[Proposition 5.4.1]{LauGL}. Let $k$ be a perfect field and $S/k$ a smooth variety. Let $f: X \to S$ a proper and smooth morphism with geometrically connected and simply connected fibers. 

\begin{prop}\label{descent_lemma_intro}
Then a non-zero perverse sheaf $L$ on $X$ is of the form $f^{\dagger}M$ iff $SS(K) \subset f^{\circ}\Lambda$, for a closed conical subset $\Lambda \subset \rT^*S$ of dimension equal to $\text{dim}(S)$. Moreover when $\text{char}(k)=0$, it suffices to assume that $SS(K) \subset f^{\circ}\rT^*S$.
\end{prop}

Using this descent criterion and an inductive argument (see Proposition \ref{key_prop_2}) we are able to show that simple objects in $\cA(Y)_{\text{Rad}}$ are in the image of the Radon transform. The inductive nature of our method naturally leads us to consider relative versions of Brylinski-Radon transforms and we develop the necessary background in Section \ref{CorrGrass}. The base case (i.e. $n-d=1$) for the induction follows from the work of Laumon \cite[Proposition 5.7]{Lau}. Finally using the isomorphism on $\text{Ext}^1$ (Corollary \ref{fully_faithfullness}) we deduce Theorem \ref{thm:mainthm}, (3).\\

We would like to note that our proof of Theorem \ref{thm:mainthm} also applies to $\ell$-adic \'etale sheaves using the notion of singular support for $\ell$-adic sheaves as described in \cite[Section 5.5]{UYZ}. It also works when $k=\bbC$ and one considers algebraically constructible sheaves in the analytic topology with Kashiwara-Schapira's \cite[Chapter V]{KS} definition of singular supports.\\



{\bf Acknowledgements:}

 We would like to thank Ahmed Abbes for his interest and encouragement during the course of this project. KVS would like to thank Ofer Gabber and Ankit Rai for useful conversations. In particular, he is thankful to Ofer Gabber for presenting a counterexample to an optimistic form of Corollary \ref{converse_proper_base_change}, ultimately resulting in the formulation of Proposition \ref{descent_perverse_sheaf}. KVS would also like to thank Hiroki Kato for patiently answering his questions about sensitivity of vanishing cycles to test functions in positive characteristics. 

\section{Background and some preliminary observations}\label{Background}

\subsection{Recollection of singular support}\label{RecSS}

Let $X$ be a smooth variety over a perfect base field $k$. Let $C \subset \rT^*X$ denote a closed conical subset, and $h: U \rightarrow X$ a morphism with $U$ smooth. Then $h$ is said to be $C$-{\it transversal} if for all geometric points $u$ of $U$, $$ker(dh_u) \cap C_{h(u)} \setminus \{0\} = \emptyset.$$ Note $C$-transversality implies that $dh|_{C \times_X U}$ is finite and Beilinson defines $h^{\circ}(C)$ to be its image in
$\rT^*U$, also a closed conical subset \cite[Section 1.2]{B}. In particular, $h^{\circ}$ always makes sense when $h$ is a smooth morphism (since such morphisms are automatically $C$-transversal for any $C$). This will be the only relevant case for us. Similarly, for any closed conical subset $C \subset \rT^*U$ whose base is proper over $X$, Beilinson defines $h_{\circ}(C)$ to be the image of $dh^{-1}(C)$ under the natural projection $\rT^*X \times_X U \to \rT^*X$. This is a closed conical subset of $\rT^*X$.\\

For any sheaf $K \in \rD^b_c(X)$, Beilinson defines the singular support $SS(K) \subset \rT^{*}(X)$. We recall some properties of $SS(K)$ which will be used in the following.

\begin{enumerate}
\item For $K \neq 0$, $SS(K)$ is a equidimensional closed conical subset of $\rT^{*}(X)$ of dimension equal to  $\dim(X)$ \cite[Theorem 1.3 (ii)]{B} .
\item Given an $SS(K)$-transversal morphism $h: U \rightarrow X$, $SS(h^*K) \subset h^{\circ}(SS(K))$ \cite[Lemma 2.2 (i)]{B}. Moreover, one has equality if $h$ is a smooth morphism \cite[Theorem 1.4, (i))]{B}.
\item Suppose $f \colon X \to Y$ is a proper morphism of smooth varieties, then for any sheaf $K$ on $X$,~$SS(f_*K) \subset f_0(SS(K))$ \cite[Lemma 2.2 (ii))]{B}.
\item $SS(K)$ is the zero section (denoted below by $0_{\rT^*X}$) iff $\mathcal{H}^{i}(K)$ are locally constant for all $i$ and atleast one of them is non-zero \cite[Lemma 2.1 (iii)]{B}.
\item For any sheaf $K$ one has $SS(K)=\bigcup_{\alpha}SS(K_{\alpha})$, where $K_{\alpha}$ runs over the various Jordan-Holder components of $\cpH^i(K)$ for every $i$ \cite[Theorem 1.4, (ii)]{B}.
\end{enumerate}

We record the following standard lemma for use below. 

\begin{lem}\label{functoriality_conic}
Let $X \xrightarrow{f} Y \xrightarrow{g} Z$ be smooth proper morphisms of smooth varieties over an algebraically closed field $k$.
\begin{enumerate}
\item Given a conic $\Lambda \subset \rT^*X$, $(g_{\circ} \circ f_{\circ})(\Lambda) = (g \circ f)_{\circ}(\Lambda)$.
\item Given a conic $\Lambda \subset \rT^*Z$, $(f^{\circ} \circ g^{\circ})(\Lambda) = (g \circ f)^{\circ}(\Lambda)$.
\item Given a conic $\Lambda \subset \rT^*Y$, $(f_{\circ} \circ f^{\circ})(\Lambda) = \Lambda$.
\item Consider a commutative square:
$$
\xymatrix{
X \ar[r]^{f} \ar[d]^{g} & Y \ar[d]^{g'}\\
Z \ar[r]^{f'}  & W
}
$$
where all morphisms and varieties are smooth proper. Then, given $\Lambda \subset \rT^*Z$, one has 
$((g')^{\circ} \circ f'_{\circ})(\Lambda) =(f_{\circ} \circ g^{\circ})(\Lambda)$

\end{enumerate}

\end{lem}
\begin{proof}
The first three parts of the lemma are immediate from the definition. Using (3) we can reduce (4) to the case when the diagram is cartesian in which case the lemma is clear.   
\end{proof}

\subsection{Relative Brylinski-Radon transform}\label{CorrGrass}

In what follows, we shall fix a base scheme $S$ which is assumed to be smooth over an algebraically closed field $k$.\\

Let $\sE$ be a vector bundle over $S$ of rank $n+1 \geq 2$. Let $0 \leq d \leq n-1$ be an integer. We denote by $\bG(d,\sE)$ the Grassmannian bundle parametrizing locally free quotients of $\sE^{\vee}$ of rank $d+1$. In particular, given an $S$-scheme $\pi: T \rightarrow S$, $\bG(d,\sE)(T)$ consists of equivalence classes of quotients $\pi^*\sE^{\vee} \rightarrow \sF$ where $\sF$ is locally free of rank $d+1$. We denote by $\pi_{d,\sE}$ the canonical morphism from $\bG(d,\sE)$ to $S$. It is a proper and smooth morphism of relative dimension $(d+1)(n-d)$.

\begin{rem}\label{symmetry_Grassmannian}
Note that we may identify $\bG(d,\sE)$ with $\bG(n-d-1,\sE^{\vee})$ by passing to duals.
Below, when working over $S = \Spec(k)$ (where $k$ is algebraically closed), we denote by $\bG(d,n)$\footnote{We use the convention that $\bG(d,n)=\emptyset$ if $d$ is negative.} the Grassmanian of $d+1$-planes in $V = k^{n+1}$. We shall also sometimes identify the latter with the $d$-planes in $\bbP^n$.
\end{rem}

The following decomposition theorem is  well-known, and is recorded here for future use.

\begin{lem}\label{decomposition_grassmannian}
For any $K \in \rD^b_c(S)$, there exists a functorial (in $K$) isomorphism 
\begin{equation}\label{map_perverse}
\bigoplus_{i=0}^{n} K\otimes R^{2i}\pi_{d,\sE}\Lambda[-2i](i) \simeq \pi_{d,\sE*} \pi_{d,\sE}^{*}K
\end{equation}
\end{lem}

\begin{proof}
Using the projection formula we may assume that $K=\Lambda$.~In this case the result is a consequence of proper base change and \cite[Théorème 1.5]{Del68} owing to the cohomology of Grassmannian satisfying hard Lefschetz (even with torsion coefficients).

\end{proof}

\subsubsection{The incidence correspondence as a Grassmannian bundle:}\label{incidence_grassmannian}
Given a pair of integers $0 \leq d_1 < d_2 \leq n-1$, we denote by $Q_{d_1,d_2,\sE} \subset \bG(d_1,\sE) \times_S \bG(d_2,\sE)$ the incidence correspondence. More precisely, given a test scheme $T$ as above, recall that an element of $\bG(d_1,\sE) \times_S \bG(d_2,\sE)(T)$ is given by a tuple (upto equivalence) $(\pi^*\sE^{\vee} \rightarrow \sF_1,\pi^*\sE^{\vee} \rightarrow \sF_2)$ where $\sF_i$ is a rank $d_i +1$ quotient. With this notation, $Q_{d_1,d_2,\sE}(T)$ consists of tuples such that there is a surjection $\sF_2 \rightarrow \sF_1$ compatible with the maps $\pi^*\sE^{\vee} \rightarrow \cF_i$. Note that if such a surjection exists, it is unique. Moreover, this is a closed subscheme of $\bG(d_1,\sE) \times_S \bG(d_2,\sE)$.\\

Let $0 \rightarrow \sS_{n-d,\sE^{\vee}} \rightarrow \pi^*_{d,\sE} \sE^{\vee} \rightarrow \sQ_{d+1,\sE^{\vee}} \rightarrow 0$ denote the universal exact sequence on $\bG(d,\sE)$. Here
$\sS_{n-d,\sE^{\vee}}$ (resp. $\sQ_{d+1,\sE^{\vee}}$) is the universal sub-bundle of rank $n-d$ (resp. quotient of rank $d+1$). With this notation, one can identify $Q_{d_1,d_2,\sE}(T)$ as the rank $n-d_2$ quotients of $\pi_T^*(\sS_{n-d_1,\sE^{\vee}}^{\vee})$, and in particular, we may view $Q_{d_1,d_2,\sE} \rightarrow \bG(d_1,\sE) $ as the Grasmmannian bundle $\bG(n-d_2-1,\sS_{n-d_1,\sE^{\vee}}^{\vee})$. By the aforementioned remark, we may also view this as the Grassmannian bundle $\bG(d_2-d_1-1,\sS_{n-d_1,\sE^{\vee}})$. In a similar manner, we may view the incidence correspondence as a Grassmannian bundle $\bG(d_1, \sQ_{d_2+1,\cE^{\vee}}^{\vee})$ over $\bG(d_2,\sE)$.\\

We denote by $p_{d_1,d_2,\sE}$ (resp. $p^{\vee}_{d_1,d_2,\sE}$, resp. $\pi_{d_1,d_2,\sE}$) the induced map from $Q_{d_1,d_2,\sE}$ to $\bG(d_1,\sE)$ (resp. $\bG(d_2,\sE)$, resp. $S$). As noted above $p_{d_1,d_2,\sE}$ (resp. $p^{\vee}_{d_1,d_2,\sE}$) is a Grassmannian bundle parametrizing locally free quotients of rank $d_2-d_1$ (resp. of rank $d_1+1$) of a vector bundle of rank $n-d_1$ (resp. of rank $d_2+1$) on $\bG(d_1,\sE)$ (resp. $\bG(d_2,\sE)$). Thus $p_{d_1,d_2,\sE}$ (resp. $p^{\vee}_{d_1,d_2,\sE}$) is proper and smooth of relative dimension
$(d_2-d_1)(n-d_2)$ (resp. $(d_1+1)(d_2-d_1)$). \\

\subsubsection{Brylinski-Radon transform}\label{Brylinski-Radon_Sheaves}We define functors $\cR_{d_1,d_2,\sE}:\rD^{b}_c(\bG(d_1,\sE)) \rightarrow \rD^{b}_c(\bG(d_2,\sE))$ and $\check{\cR}_{d_1,d_2,\sE}:\rD^{b}_c(\bG(d_2,\sE)) \rightarrow \rD^{b}_c(\bG(d_1,\sE))$ as follows,

\begin{equation}
\cR_{d_1,d_2,\sE}(K):=p^{\vee}_{d_1,d_2,\sE*}p^{\dagger}_{d_1,d_2,\sE}K
\end{equation}

\noindent and

\begin{equation}
\check{\cR}_{d_1,d_2,\sE}(L):=p_{d_1,d_2,\sE*}p^{\vee \dagger}_{d_1,d_2,\sE}L.
\end{equation}

Finally, we make explicit a condition on closed conical subsets of $\rT^*\bG(d,\sE)$ (resp. $\rT^*Q_{d_1,d_2,\sE}$) which will be important in the following\footnote{See Example \ref{failure_descent} for a motivation to consider the condition ($\ast$).}.

\begin{defn}\label{regular_conical}
We will say that a closed conical subset $C \subset \rT^*\bG(d,\sE)$  (resp. $\rT^*Q_{d_1,d_2,\sE}$) is {\it regular over $S$} (or just {\it regular} if $S$ is clear from context) if the following condition is satisfied:
\begin{enumerate}[($\ast$)]
    \item Every irreducible component $\Lambda$ of $C$  contained in $\pi_{d,\sE}^{\circ}\rT^*S$ (resp. $\pi_{d_1,d_2,\sE}^{\circ}\rT^*S$) is of the form $\pi_{d,\sE}^{\circ}\Lambda'$ (resp. $\pi_{d_1,d_2,\sE}^{\circ}\Lambda'$)  for an irreducible closed conical subset $\Lambda' \subset \rT^*S$.
\end{enumerate}
\end{defn}

Note that condition ($\ast$) above is trivially satisfied when $S=\Spec(k)$ and $C$ is of pure dimension of dimension equal to $\text{dim}(\bG(d,\sE))$ (resp. $\text{dim}(Q_{d_1,d_2,\sE})$).

Let $C\subset \rT^*\bbP(\sE)$ be a closed conical subset. We denote by 

\begin{equation}\label{Radon_conical}
    \text{Rad}_{0,d,\sE}(C):=(p^{\vee}_{0,d,\sE})_{\circ}(p_{0,d,\sE})^{\circ}(C),
\end{equation}

\noindent the Radon transform of $C$ with respect to $R_{0,d,\sE}$. This is a closed conical subset of $\rT^*\bG(d,\sE)$.\\

Let $q_{0,d,\sE}$ and $q_{0,d,\sE}^{\vee}$ denote the morphism from $\rT^*_{Q_{0,d,\sE}}(\bbP(\sE)\times_S\bG(d,\sE))$ to $\rT^*\bbP(\sE)$ and $\rT^*\bG(d,\sE)$ respectively. We need the following, which is the relative version of \cite[Lemme 5.6]{Br}, and follows from it.

\begin{lem}\label{radon_conic_lemma} 
Let $\dot{q}_{0,d,\sE}$ and $\dot{q}_{0,d,\sE}^{\vee}$ respectively be the induced morphisms from $\rT^*_{Q_{0,d,\sE}}(\bbP(\sE)\times_S\bG(d,\sE)) \backslash \pi_{0,d,\sE}^{\circ}\rT^*S$ to $\rT^*\bbP(\sE) \backslash \pi_{0,\sE}^{\circ}\rT^*S$ and $\rT^*\bG(d,\sE) \backslash \pi_{d,\sE}^{\circ}\rT^*S$. Then

\begin{enumerate}[(a)]
    \item $\dot{q}_{0,d,\sE}$ is smooth and proper of relative dimension $d(n-d-1)$.
    \item $\dot{q}_{0,d,\sE}^{\vee}$ is a closed immersion.
\end{enumerate}
\end{lem}

As a consequence we have the following.

\begin{cor}\label{corollary_conic_indcidence}
Let $C \subset \rT^*Q_{0,d,\sE}$ be a closed conical subset. Suppose $C=p_{0,d,\sE}^{\circ}C_1=p_{0,d,\sE}^{\vee \circ}C_2$ for closed conical subsets $C_1$ and $C_2$ in $\rT^*\bbP(\sE)$ and $\rT^*\bG(d,\sE)$ respectively. Then $C \subset \pi_{0,d,\sE}^{\circ}\rT^*S$.
\end{cor}

\begin{rem}\label{dual_statement_corollary}
    Note that by Remark \ref{symmetry_Grassmannian} the above corollary is also true for correspondences between $\bG(d,\sE)$ and $\bG(n-1,\sE)$ with $d<n-1$.
\end{rem}

\begin{cor}\label{criterion_descent_S}
Let $M$ be perverse sheaf on $Q_{d,n-1,\sE}$ (with $d<n-1$) that belongs to both $p_{d,n-1}^{\dagger} \cP(\bG(d,\sE))$ and $p_{d,n-1}^{\vee \dagger} \cP(\bG(n-1,\sE))$. Then $SS(M)\subset \pi_{Q_{d,n-1,\sE}}^{\circ}\rT^*S$.
\end{cor}

\begin{proof}
The corollary is an immediate consequence of the remark above and Section \ref{RecSS}, (2).
\end{proof}

We also note the following corollary.

\begin{cor}\label{corollary_radon_conic_lemma}
 Let $C \subset \rT^*\bbP(\sE)$  be a closed conical subset regular over $S$. Then $\text{Rad}_{0,d,\sE}(C)$ is also regular over $S$.
\end{cor}

\begin{proof}
Let $\Lambda \subset C$ be an irreducible component of the form $\pi_{0,\sE}^{\circ}\Lambda'$. Then Lemma \ref{functoriality_conic} implies that $\text{Rad}_{0,d,\sE}(\Lambda)=\pi_{d,\sE}^{\circ}\Lambda'$. On the other hand if $\Lambda$ is not contained in $\pi_{0,\sE}^{\circ}\rT^*S$, then Lemma \ref{radon_conic_lemma} implies that $\text{Rad}_{0,d,\sE}(\Lambda)$ is an irreducible component of $\text{Rad}_{0,d,\sE}(C)$ and is not contained in $\pi_{d,\sE}^{\circ}\rT^*S$.
\end{proof}

\section{Proof of Theorem 1: Preliminary Results}
In this section, we collect some results which will be used in the following for the proof of part (3) of Theorem \ref{thm:mainthm}.

\subsection{A criterion for descent of perverse sheaves.}
As before, let $k$ be an algebraically closed field, $S/k$ be a smooth variety and let $f \colon X \to S$ be a smooth morphism whose fibres are connected of dimension $d$. 
In general, it is hard to characterise the subcategory $f^{\dagger}\cP(S)$ of $\cP(X)$. If, in addition to the above assumptions, $f$ is proper and the fibres of $f$ are simply connected, then we have the following descent criterion.

\begin{prop}\footnote{Our proof also works when $k$ is only assumed to be perfect, provided $f$ is  \textit{geometrically} connected.}\label{descent_perverse_sheaf}
A (non-zero) simple perverse sheaf $K \in \cP(X)$ is in the essential image of $f^{\dagger}$ iff $SS(K) \subseteq f^{\circ}\Lambda$, for some closed conical subset $\Lambda \subset T^*S$ of dimension equal to $\text{dim}(S)$. Moreover, when $\text{char}(k)=0$ it suffices to assume that $SS(K) \subset f^{\circ}T^*S$.
\end{prop}

\begin{proof}
Since $f$ is smooth, the necessity results from the preservation of singular supports under pullback (see Section \ref{RecSS}, (2)). Suppose now that $K$ is a (non-zero) simple perverse sheaf on $X$ such that  as in $SS(K) \subseteq f^{\circ}\Lambda$, with $\Lambda$ as in the proposition. Since $K$ is simple, there exists a triple $(X',U,\sL)$ consisting of an irreducible closed subset $X' \xhookrightarrow{i} X$, a non-empty smooth (over $k$) open subset $U \xhookrightarrow{j} X'$ and a non-zero irreducible local system $\sL$ on $U$ such that $K=i_*j_{!*}\sL[\text{dim}(X')]$ \cite[Théorème 4.3.1, (ii)]{BBDG18}. Note that $f^{\circ}$ preserves irreducible components since $f$ is smooth. As a consequence, by removing any extra components (if necessary), we may assume that  $SS(K)=f^0\Lambda$.\\

\noindent{\it Claim 1:} It is sufficient to prove the theorem after replacing $S$ by an open dense subset $S' \xhookrightarrow{j'} S$, $X$ by $X_{S'} := X \times_S S'$, and $K$ by $K|_{X_{S'}}$ provided $K|_{X_{S'}}$ is non-zero.\\
{\it Proof:} Let $j'': X_{S'} \hookrightarrow X$ denote the resulting open immersion. First note that the resulting map $f':X_{S'} \rightarrow S'$ satisfies the hypotheses of the theorem, and $SS(K|_{X_{S'}}) = SS(K)|_{X_{S'}} = (f')^{\circ}(\Lambda|_{S'})$ (Section \ref{RecSS}, (2)). If $M$ is a simple perverse sheaf on $S'$ such that $(f')^{\dagger}M = K|_{X_{S'}}$, then $f^{\dagger}j'_{!*}(M) = j''_{!*}((f')^{\dagger}M) = j''_{!*}(K|_{X_{S'}}) = K$. Here the first equality follows from the fact that intermediate extensions commute with pull back along smooth morphisms \cite[Lemme 4.2.6.1]{BBDG18}, and the last follows from the fact that $K$ is a simple perverse sheaf.\\

\noindent{\it Claim 2:} We may assume that the base $S'$ of $\Lambda$ is smooth, $X' = X \times_S S'$, and $SS(K) = f^{\circ}(\Lambda)$.\\
{\it Proof:} Let $S'$ be the base of $\Lambda$. Since the base of $SS(K)$ equals the support of $K$ \cite[Lemma 2.3 (iii)]{B} we have $X'=f^{-1}(S')$. Let $Z$ denote the singluar locus of $S'$. Since $k$ is algebraically closed, this is a strict closed subset of $S'$. In particular, $S \setminus Z$ is open, and by the previous claim, we may base change everything to $S \setminus Z$. \\


\noindent{\it Claim 3:} Let $\Lambda'$ be an irreducible component of $f^{\circ}\Lambda$ which is not equal to $\rT^*_{X'}X$, the conormal bundle of $X'$ in $X$. Then the base of $\Lambda'$ does not dominate $S'$. In particular, the union of the bases of the components of $SS(K)$ not equal to $\rT^*_{X'}X$ (denoted by  $X''$ below) cannot dominate $S'$ under $f$.\\
{\it Proof:} Let $Z \subset X'$ be the base of $\Lambda'$. We claim that $Z$ does not dominate $S'$ under $f$. First note that, if  $Z \neq X'$, then it does not dominate $S'$. We're reduced to showing that if $Z = X'$, then $\Lambda' = \rT^*_{X'}X$. Since $X'$ is smooth and $K=i_*j_{!*}\sL[\text{dim}(X')]$, $SS(K)=i_0SS(j_{!*}\sL[\text{dim}(X')])$ (Combine \cite[Lemma 2.5 (i)]{B} and \cite[Theorem 1.5]{B}). Note that $i_0$ preserves bases of irreducible components, and there exists a unique component of $SS(j_{!*}\sL[\text{dim}(X')])$ whose base equals $X'$ (namely the zero section). It follows that there is a unique component of $SS(K)$ whose base is $X'$ (namely $\rT^*_{X'}X$).\\


\noindent  Note that $X''=f^{-1}(f(X'')) \subsetneq X$. Let $U'=X' \backslash X''$, then $f|_{U'} \colon U' \to S' \backslash f(X'')$ is a proper morphism with connected and simply connected fibres. Thus by \cite[Expos\'e X, Corollaire 2.2]{SGA1} there exists a local system $\sM$ on $S' \backslash f(X'')$ such that $f|_{U'}^*\sM=\sK$.~Thus by uniqueness $K=f^{*}(i_{S'*}j_{U'!*}(\sM[\text{dim}(S)])$,~here $i_{S'}$ (resp. $j_{U'}$) are the immersions from $S'$ (resp. $U'$) into $S$ (resp. $S'$).\\

\noindent Now suppose $\text{char}(k)=0$, then every irreducible component (say $\tilde{\Lambda}$) of $SS(K)$ is Lagrangian \cite[Proposition 2.2.7]{Sai} and further the smooth locus of $\tilde{\Lambda}$ is the conormal to the smooth locus in the intersection of $\tilde{\Lambda}$ with the zero section of $\rT^*X$ (\cite{B}, Exercise in Section 1.3). Such a component $\tilde{\Lambda}$ is in $f^0\rT^*S$ iff it is the inverse image of a closed conical subset of $\rT^*S$.

\end{proof}

\begin{rem}\label{in_Char_p}
It follows from the proof of Proposition \ref{descent_perverse_sheaf} that even in positive characteristic, as long as the components of the singular support are conormals (and not just Lagrangians!), the apparently weaker assumption $SS(K) \subset f^{\circ}T^*S$ suffices.
\end{rem}

While the following corollary will not be used in what follows, we record it here since it may be of independent interest.

\begin{cor}\label{converse_proper_base_change}
Let $f \colon X \to S$ and $K$ be as in Proposition \ref{descent_perverse_sheaf}. Then $K$ is lisse iff $\cpH^{d}f_*K$ is lisse. 
\end{cor}

We continue using the notation from Proposition \ref{descent_perverse_sheaf}. We record below an example which shows that if $\text{char}(k)>0$, it is in general not sufficient to assume $SS(K) \subset f^0T^*S$. 

\begin{example}\label{failure_descent}
Let $k$ be a perfect field of characteristic $p>0$. Let $S=\bbA^1_s$, $X=\bbA^1_s \times \bbP^1_{[t:t']}$\footnote{We use subscripts to denote a choice of a coordinate system}, and $f \colon X \to S$ the projection map. Let $\tilde{X}:=Z(t'^p(x^{p^2}-x)-(s+x^p)t^p) \subseteq \bbA^1_{x} \times \bbA^{1}_s \times \bbP^1_{[t:t']}$ and denote by $\pi: \tilde{X} \to X$ the induced map.~We denote by $\tilde{X}_{t \neq 0}$ (resp. $X_{t \neq 0}$) and $\tilde{X}_{t' \neq 0}$ (resp. $X_{t' \neq 0}$) the open cover of $\tilde{X}$ (resp. $X$) obtained from the usual cover on $\bbP^1_{[t:t']}$.

Note that $\tilde{X}$ is a smooth surface over $k$ and that $\pi$ is finite \'etale of rank $p^2$ over $X_{t' \neq 0}$.~Over the line $t'=0$,~it is a totally ramified cover of $\bbA^1_s$. Thus $\pi$ is finite. and we denote by $K=\pi_*(\Lambda[2])$ and thus by Section \ref{RecSS}, (3), $SS(K) \subseteq \pi_{\circ}(0_{\rT^*\tilde{X}})$.

It follows from the definition of $\pi_{\circ}$ that $\pi_{\circ}(0_{\rT^*\tilde{X}})=0_{T^*X} \cup \Lambda$.~Here $\Lambda$ is $f^{\circ}\rT^*\bbA^1_s|_{t'=0}$.~By proper base change $K$ is not a lisse perverse sheaf, hence $SS(K)=0_{\rT^*X} \cup \Lambda$. Moreover, $K$ is not the pullback of a perverse sheaf from $\bbA^1_s$, since if that were the case then its restriction to $s=0$ would have to be trivial by proper base change. This in turn implies that the finite \'etale cover $\tilde{X}_{t' \neq 0} \to X_{t' \neq 0}$ is trivial restricted to $s=0$, which is not the case by the choice of the Artin-Schrier cover.

\end{example}

\subsection{A key proposition}\label{key_prop}
In this section, we prove a key proposition which will be used in the proof of Theorem \ref{thm:mainthm}, (3). Recall we have a base scheme $S$ smooth over $k$ ( assumed to be algebraically closed) and a vector bundle $\sE$ on $S$ of rank $n+1$. We continue using the notations from Section \ref{CorrGrass}. However, for ease of exposition, we drop $\sE$ from the notation. In particular we shall denote $\bG(0,\sE)$ by $\bbP$, $\bG(d,\sE)$ by $\bG(d)$ and $\bG(n-1,\sE)$ by $\bG(n-1)$.

Below, we shall makes use of the following commutative diagram in order to facilitate an inductive argument.

\begin{equation}\label{key_diagram}
\xymatrix{Q_{0,d,n-1} \ar[rr] \ar[dr] \ar@{.>}[dd] & & Q_{d,n-1} \ar[dd] \ar[dr] & \\
              & Q_{0,d} \ar[rr] \ar[dd] &  & \bG(d) \ar[dd] \\
          Q_{0,n-1} \ar[rr] \ar[dr]& & \bG(n-1) \ar[dr] & \\
              & \bbP \ar[rr] & & S}.          
\end{equation}

In diagram (\ref{key_diagram}), the bottom, front and right hand side faces are the correspondences described in Section \ref{CorrGrass}. We \textit{define} $Q_{0,d,n-1} := Q_{0,d} \times_{\bG(d)} Q_{d,n-1}$. This induces a morphism from $Q_{0,d,n-1}$ to $\bbP \times_S \bG(n-1)$, which by construction factors through $Q_{0,n-1}$ (denoted in the diagram (\ref{key_diagram}) by the dotted arrow). We have the following lemma which follows from the description of the incidence correspondence as a Grassmannian bundle in Section \ref{incidence_grassmannian}. 

\begin{lem}\label{IndLem}
There exists isomorphisms (as $\bG(n-1)$-schemes) $Q_{0,n-1} \simeq \bbP(\sQ^{\vee}_{n,\sE^{\vee}})$, $Q_{d,n-1} \simeq \bG(d,\sQ^{\vee}_{n,\sE^{\vee}})$ and $Q_{0,d,n-1} \simeq Q_{0,d,\sQ^{\vee}_{n,\sE^{\vee}}}$ such that commutative square 

\begin{equation}\label{IndLem1}
\xymatrix{Q_{0,d,n-1} \ar[r] \ar[d] & Q_{d,n-1} \ar[d] \\
          Q_{0,n-1} \ar[r] & \bG(n-1)},
\end{equation}

\noindent in diagram (\ref{key_diagram}) is the one induced by the correspondence $Q_{0,d,\sQ^{\vee}_{n,\sE^{\vee}}} \subset \bbP(\sQ^{\vee}_{n,\sE^{\vee}}) \times_{\bG(n-1)} \bG(d,\sQ^{\vee}_{n,\sE^{\vee}})$. 

\end{lem}

\begin{proof}
Note that $Q_{0,d,n-1}=Q_{d,n-1} \times_{(\bG(d) \times_S\bG(n-1))}(Q_{0,d}\times_S\bG(n-1))$. Thus in order to prove the lemma, it suffices to show that projective sub-bundle of $\bbP\times_S\bG(n-1)$ defined by $Q_{0,n-1}$ induces the  Grassmannian sub-bundle $Q_{d,n-1}$ of $\bG(d)\times_S\bG(n-1)$. But this follows from the description in Section \ref{incidence_grassmannian}.

More precisely using the notations from the section, $Q_{0,n-1}$ is the projective bundle (over $\bG(n-1)$) defined by the sub-bundle $\sQ^{\vee}_{n,\sE^{\vee}}$ of $\pi_{n-1,\sE}^*\sE$ and $Q_{d,n-1}$ is the Grassmannian bundle $\bG(d,\sQ^{\vee}_{n,\sE^{\vee}})$.
\end{proof}

In what follows we denote the vector bundle $\sQ^{\vee}_{n,\sE^{\vee}}$ on $\bG(n-1)$ by $\sF$. In particular there is a Radon transform (denoted by $\cR_{0,d,\sF}$) from $\rD^b_c(Q_{0,n-1})$ to $\rD^b_c(Q_{d,n-1})$. The following lemma is an immediate consequence of proper base change applied to the cartesian square at the top of Diagram (\ref{key_diagram}).

\begin{lem}\label{key_BC}
For any perverse sheaf $K$ on $\bbP$ we have $\cR^i_{0,d,\sF}(p_{0,n-1}^{\dagger}K) \isom p^{\dagger}_{d,n-1}\cR^i_{0,d}(K)$\footnote{Here and in the rest of this article by $\cR^{i}_{d_1,d_2,\sE}$ we mean $\cpH^{i} \circ \cR_{d_1,d_2,\sE}$. We use a similar convention for $\cR^{\vee}_{d_1,d_2,\sE}$.}  in $\cP(Q_{d,n-1})$.
\end{lem}

 Below, for $X$ smooth over $k$ and $\Lambda \subset \rT^*X$ is a conical subset, then
$\cP(X,\Lambda)$ is the full subcategory of the category of perverse sheaves $K$ such that $SS(K) \subset \Lambda$. Note that this is  is a Serre subcategory (see Section \ref{RecSS}, (5)).

Let $C \subset \rT^*\bbP$ be a closed conical subset equidimensional of dimension equal to $\text{dim}(\bbP)$. For the rest of this section, we assume that closed conical subsets are regular over the base $S$ (see Definition \ref{regular_conical}).

\begin{prop}\label{key_prop_2}
With notation as above, any simple perverse sheaf $L$ in $\cP(\bG(d),\text{Rad}_{0,d}(C))$ is either in $\pi_{d}^{\dagger}(\cP(S))$ or there exists a simple perverse sheaf $K$ on $\bbP$ and a (decreasing) filtration $F^{\cdot}\cR^0_{0,d}K$ on $\cR^0_{0,d}K$ such that
\begin{enumerate}[(a)]
    \item $SS(K) \subseteq C$.
    \item $F^iR^0_{0,d}K=R^0_{0,d}K$ for $i \leq 0$.
    \item $F^iR^0_{0,d}K=0$ for $i \geq 3$.
    \item $\text{Gr}^i_F(\cR^0_{0,d}K)$ belongs to $\pi_{d}^{\dagger}\cP(S)$ for $i=0,2$ and $\text{Gr}^1_F(\cR^0_{0,d}K)=L$.
\end{enumerate}
\end{prop}

\begin{proof}
We may assume $L$ does not belong to $\pi_{d}^{\dagger}(\cP(S))$. We prove the claim by descending induction on $n-d$ (over varying choices of $(S,\sE)$). Suppose $n-d=1$ and hence $\bG(d)=\bbP(\sE^{\vee})$. Then (b)-(d) follow immediately from \cite[Corollaire 5.8, (i)]{Lau}. Moreover, (a) follows from the fact that $K$ is in fact a sub-quotient of $\check{\cR}_{0,n-1}(L)$.

Now suppose the Proposition has been verified for $n-d=r \geq 1$ and for all possible choices of $(S,\cE)$. We shall now prove it for $n-d=r+1$ by induction via Diagram (\ref{key_diagram}). By the induction hypothesis, we may assume that the Proposition has been verified for $\cR_{0,d,\sF}$.

It follows from \cite[Corollaire 4.2.6.2]{BBDG18} that $L_{\sF}:=p_{d,n-1}^{\dagger}L$ is simple and by Section \ref{RecSS}, (2) that $SS(L_{\sF})=p_{d,n-1}^{\circ}SS(L)$. Thus by Lemma \ref{functoriality_conic}, $SS(L_{\sF})$ is contained in the Radon transform of $p_{0,n-1}^0C$ with respect to $R_{0,d,\sF}$. Moreover by Corollary \ref{criterion_descent_S} it follows that $L_{\sF}$ is not in the essential image of $p^{\vee \dagger}_{d,n-1}\cP(\bG(n-1))$. Now by induction hypothesis there exists a simple perverse sheaf $K_{\sF}$ on $Q_{0,n-1}$ with and a filtration $F^{\cdot}_{\sF}\cR^0_{0,d,\sF}$ such that

\begin{enumerate}[(a')]
    \item $SS(K_{\sF}) \subseteq p_{0,n-1}^0C$.
    \item $F_{\sF}^iR_{0,d,\sF}^0K_{\sF}=R_{0,d,\sF}^0K_{\sF}$ for $i \leq 0$.
    \item $F_{\sF}^iR_{0,d,\sF}^0K_{\sF}=0$ for $i \geq 3$.
    \item $\text{Gr}^i_{F_{\sF}}(R_{0,d,\sF}^0K_{\sF})$ belongs to $p_{d,n-1}^{\vee \dagger}\cP(\bG(n-1))$ for $i=0,2$ and $\text{Gr}^1_{F_{\sF}}(R_{0,d,\sF}^0K_{\sF})=L_{\sF}$.
    
\end{enumerate}

Now using Proposition \ref{descent_perverse_sheaf}, (a') above implies that $K_{\sF}$ descends to a simple perverse sheaf $K$ on $\bbP$ such $SS(K) \subseteq C$. Moreover by Lemma \ref{key_BC}, $R_{0,d,\sF}^0K_{\sF}$ is in the essential image of $p_{d,n-1}^{\dagger}\cP(\bG(d))$. Thus by \cite[Section 4.2.6]{BBDG18} so are $\text{Gr}^i_{F_{\sF}}(R_{0,d,\sF}^0K_{\sF})$ for all $i$. Thus by Corollary \ref{criterion_descent_S} and Proposition \ref{descent_perverse_sheaf}, $\text{Gr}^i_{F_{\sF}}(R_{0,d,\sF}^0K_{\sF})$ for $i=0,2$ belongs to $(\pi_{d} \circ p_{d,n-1})^{\dagger} \cP(S)$. Hence the result.

\end{proof}

\section{Proof of Theorem  \ref{thm:mainthm}, (1) }

In the rest of this article we work over $S=\Spec(k)$, with $\sE$ a vector space over $k$ of dimension $n+1$ (which we henceforth ignore from the notation) and use the following notation. 

\begin{nota}\label{Not1}
We will only consider the Brylinksi-Radon transform between $\bbP$ to $\bG(d)$. 
\begin{enumerate}[(1)]
\item We will denote $\bG(d)$ by $Y$ and the incidence correspondence $Q_{0,d}$ by $Q$. The projections from $Q$ to $\bbP$ (resp.~$Y$) are denoted by $p_1$ (resp. $p_2$).

\item The morphism from $\bbP$ (resp.~$Y$) to $\Spec(k)$ are denoted by $\pi_{\bbP}$ (resp.~$\pi_{Y}$). 

\item The Brylinski-Radon transforms are denoted by $\cR$ and $\check{\cR}$.

\item Let $E$ be the complement of the incidence variety $Q \subset \bbP \times Y$. Let $p_1^{\circ}$ and $p_2^{\circ}$ be the projections to $\bbP$ and $Y$ respectively from $E$. 

\item  In what follows we will need the \textit{modified Brylinski-Radon transform} defined as $\cR_{!}K :=p^{\circ}_{2!}p_1^{\circ\dagger}K$.

\item For a complex $K$ on $\bbP$,~by $\sK$ we mean the complex $\pi_{\bbP*}K$ on $\Spec(k)$. Similarly, for complexes $K$ on $Y$.

\item We will use $\cR^i(K)$ (resp.~$\cR_{!}^i(K)$, $\sK^{i}$) to denote the $i^{\mathrm{th}}$ perverse cohomology of $\cR(K)$ (resp.~$\cR_{!}(K)$, $\sK$).

\end{enumerate}
\end{nota}

\subsection{Some preliminary observations}

The next two lemmas are immediate consequences of the smoothness and properness of $p_1$ and $p_2$, and we state them without a proof.

\begin{lem}\label{behaviour_weights_duality}
For any sheaf $K \in \rD^{b}_c(\bbP)$ and $L \in \rD^b_{c}(Y)$,~$D(\cR(K)) \simeq \cR(DK)(d(n-d))$  and $D(\check{\cR}(L))=\check{\cR}(DL)(d)$ \footnote{Here $D$ is the Verdier duality functor.}. 
\end{lem}

\begin{lem}\label{adjunction}
The functors $(\check{\cR}[\delta](d(n-d)),\cR,\check{\cR}[-\delta](d))$\footnote{In what follows we set $\delta:=d(n-d-1)$} form an adjoint triple.
\end{lem}

The following result is due to Brylinski \cite[5.3.1 (i), (ii)]{Br}. Again, while this is proved in loc. cit. in the complex analytic setting, the same proof goes through in our setting. 

\begin{prop}\label{radconst}
Let $\cR$ and $\check{\cR}$ be as before. Then $\cR$ and $\check{\cR}$ preserve the localizing set $T$ (see Section \ref{localization_category}), and in particular one has induced functors $\cR: \rD^b_c(\bbP)_T \rightarrow \rD^b_c(Y)_T$ and $\check{\cR}: \rD^b_c(Y)_T \rightarrow \rD^b_c(\bbP^n)_T.$
\end{prop} 

\subsection{An application of Artin vanishing}
We now record the following easy consequence of Artin vanishing which is used in the proof of Theorem \ref{thm:mainthm}, (1).

\begin{lem}\label{perversity_bound_lem}
Let $X/k$ be a base scheme.~Let $U$ be the complement in $\bbP^{n}_X$ of a linear subspace\footnote{A linear subspace of $\bbP^n_X$ is a closed subscheme, which Zariski locally over $X$ isomorphic to $\bbP^d_X \subset \bbP^n_X$ embedded linearly.} $Z$ of relative dimension $d$,~and let $\pi$ be the map from $U$ to $X$.~Then $\pi_{*}$ maps $^{p}\rD^{\leq 0}(U)$ to $^{p}\rD^{\leq n-d-1}(X)$.
\end{lem}

\begin{proof}
The proof is via a repeated application of Artin vanishing in the form of right t-exactness (for the perverse t-structure) of affine morphisms \cite[Th\'eor\`eme 4.4.1]{BBDG18}.~After replacing $X$ with a suitable Zariski open we can consider a chain of linear subspaces $Z_0 \subsetneq Z_1 \subsetneq \cdots Z_{n-d-1}$ of $Z$ such that $Z_0=\bbP^d_X$ and $\text{dim}(Z_i)=d+i$.~Let $U_i :=\bbP^n_X \backslash Z_i$ be the corresponding open subscheme.~Let $\pi_i$ be the map from $U_i$ onto $X$,~and we identify $\pi_0$ with $\pi$.

We prove the lemma by descending induction on $i$.~For $i=n-d-1$ the lemma is an immediate consequence of Artin vanishing \cite[Th\'eor\`eme 4.1.1]{BBDG18}.~Assuming that the lemma has been verified up to some $i \leq n-d-1$,~we prove it for $i-1$.~Let $j$ (resp. $l$ ) be the inclusion of $U_{i}$ (resp. $Z_i \backslash Z_{i-1}$) inside $U_{i-1}$.~Let $K$ be a sheaf on $U_{i-1}$ in $^{p}\rD^{\leq 0}(U_{i-1})$.~By induction hypothesis $\pi_{*}(j_*j^*K) \in~^{p}\rD^{\leq n-d-1-i}(X)$.~Thus it suffices to show $\pi_*(l_*l^!K) \in~^{p}\rD^{\leq n-d-i}(X)$.

By construction $Z_i \backslash Z_{i-1}$ is at once  affine over $X$ and a complete intersection of codimension $n-d-i$ in $U_{i-1}$,~and thus \cite[Corollaire 4.1.10,~(ii)]{BBDG18} implies the result.

\end{proof}

 The following corollary will be used below to describe the image of the Brylinski-Radon transform.

\begin{cor}\label{perversity_bound_radon}
With notation as above, $p^{\circ}_{2!}$ maps $^{p}\rD^{\geq 0}(E)$ to $^{p}\rD^{\geq -(n-d-1)}(Y)$.
\end{cor}

\subsection{Proof of \ref{thm:mainthm}, (1) and Corollaries}
In fact, we prove the following more refined version of Theorem  \ref{thm:mainthm}, part (1).

\begin{thm}\label{image_Grassmannian_Brylynski_Radon}
Let $K$ be a sheaf on $\bbP$.

\begin{enumerate}[(1)]

\item If $K$ is upper semi-perverse then for any $i<0$, we have $\cR^i(K) \simeq \pi_{Y}^{\dagger}\sK^{i-n+d}$. 

%

 \item If $K$ is perverse,~$\cR^i(K)$ are constant for any $i \neq 0$.~Also the perverse sheaves $\cR^i_!(K)$ are constant for $i \neq n-d+1$.

 \item Consequently $\cR$ is $t$-exact for the perverse $t$-structures on $\rD^b_c(\bbP)_T$ and $\rD^b_c(Y)_T$ (see Section \ref{localization_category}).

\end{enumerate}

\end{thm}

\begin{proof}

By definition of $\cR$ (and $\cR_!$) and proper base change, we have a triangle on $Y$

\begin{equation}\label{basic_triangle_radon}
\xymatrix{\cR(K)[n-d-1] \ar[r] & \cR_!K \ar[r] & \pi_{Y}^*\cK[(d+1)(n-d)] \ar[r]^-{+1} &}.
\end{equation}
Now, by Corollary \ref{perversity_bound_radon} and \cite[Section 4.2.4]{BBDG18}, one has that for any $K \in~^p\rD^{\geq 0}(\bbP)$, $\cR_!K\in~^p\rD^{\geq -(n-d-1)}(Y)$.~Taking the long exact sequence of perverse cohomologies associated to the triangle (\ref{basic_triangle_radon}) gives us (1).

If $K$ is perverse, then applying the first part to $DK$ and using Lemma \ref{behaviour_weights_duality} we deduce (2).~The constancy of $\cR_!^i(K)$ for $i \neq n-d+1$ then follows from the fact that constant sheaves form a Serre subcategory. The $t$-exactness of $\cR$ is now clear.

\end{proof}

We get the following corollaries by combining Lemma \ref{adjunction} and Theorem \ref{image_Grassmannian_Brylynski_Radon}.

\begin{cor}\label{adjunction_t_structure}
The functor $\check{\cR}[-\delta](d)$ (resp. $\check{\cR}[\delta](d(n-d))$) is left $t$-exact (resp. right $t$-exact) for the perverse $t$-structures on $\rD^b_c(Y)_T$ and $\rD^b_c(\bbP)_T$.
\end{cor}

\begin{cor}\label{adjunct_triple_perverse}
$(\check{\cR}^{\delta}(d(n-d)),\cR^{0}, \check{\cR}^{-\delta}(d))$\footnote{We denote $\cpH^i_T \circ R$ by $\cR^i$ and a similar notation for $\check{\cR}^i$.} form an adjoint triple between $\cA(\bbP)$ and $\cA(Y)$.~Moreover $\check{\cR}^{-\delta}(d)$ (resp. $\check{\cR}^{\delta}(d(n-d))$) is left t-exact (resp. right t-exact).
\end{cor}

\section{Proof of Theorem \ref{thm:mainthm}, (2) and (3)}

In this section, we prove Theorem \ref{thm:mainthm}, (2) and (3).

\subsection{Proof of Theorem \ref{thm:mainthm}, (2) and corollaries}
Consider the following diagram of schemes, where the central square is cartesian by definition:
\begin{equation}\label{Radon_inverse_Radon_1}
\xymatrix{  &    & Q \times_{Y} Q \ar[dl]_-{\tp_{2}} \ar[dr]^-{\tp_{2}} & & \\
                   &  Q \ar[dl]_-{p_1} \ar[dr]^-{p_{2}} & &  Q \ar[dl]_-{p_{2}} \ar[dr]^-{p_{1}} & \\
                 \bbP & & Y & & \bbP}.
\end{equation}

Let $\pi: Q \times_{Y} Q \to \bbP \times \bbP$ denote the morphism induced by $p_1$ on each factor. Let $s_1: \bbP \times \bbP \rightarrow \bbP$ (resp. $s_2: \bbP \times \bbP \rightarrow \bbP$) be the projection onto the first (resp. second) factor. An application of proper base change along the central cartesian square in diagram (\ref{Radon_inverse_Radon_1}) and the projection formula gives a natural (in $\cK$) isomorphism:
\begin{equation}\label{Radon_inverse_Radon_2}
\check{\cR} \circ \cR(K)=s_{2*} \left (s_1^{*}K \otimes_{\Lambda} \pi_{*}\Lambda[\delta_{+}] \right)\footnote{In what follows we set $\delta_{+}:=d(n-d+1)$.}.
\end{equation}
Let $\Delta: \bbP \hookrightarrow \bbP \times \bbP$ denote the diagonal embedding, let $U$ be the complement of the diagonal embedding, and let $j \colon U \hookrightarrow \bbP \times \bbP$ be the corresponding open immersion. One has the resulting diagram with cartesian squares:
$$
\xymatrix{
Q \ar[r]^-{i_Q} \ar[d]^{p_1} & Q \times_Y Q \ar[d]^{\pi} & W \ar[l]_-{j_W} \ar[d]^{\pi_U}  \\
\bbP \ar[r]^-{\Delta}  & \bbP \times \bbP & U \ar[l]_-{j} }.
$$

We note that $\pi_U$ is a Grassmann bundle with fibers $\bG(d-2,n-2)$. Consider the natural closed immersion $Q \times_Y Q \rightarrow Q \times \bbP$, which on closed points maps $(x,y,L)$ to $(x,L,y)$. Here $x,y$ are closed points of $\bbP$ and $L \subset \bbP$ is a $d$-plane containing them. The above commutative diagram factors as:

$$
\xymatrix{
Q \ar[r]^-{i_Q} \ar[d]^{Id} & Q \times_Y Q \ar[d]  & W \ar[l]_-{j_U} \ar[d]\\
Q \ar[r] \ar[d] & Q \times \bbP \ar[d]^{\tilde{\pi}} & V  \ar[l] \ar[d]^{\tilde{\pi}_{U}} \\
\bbP \ar[r]^-{\Delta}  & \bbP \times \bbP & U \ar[l]_-{j}},
$$

\noindent where all the squares are Cartesian.

Note that $\tilde{\pi}$ is a Grassmannian bundle with fibers $\bG(d-1,n-1)$ and is identity along the second projection. Let $Z:=  V \setminus W=Q \times \bbP \setminus Q \times_Y Q$, and $\pi_Z: Z \rightarrow U$ denote the resulting morphism.
%
%
%
%
%
%
%
%
We have an exact triangle on $U$

\begin{equation}\label{M4}
\xymatrix{ \pi_{Z!}\Lambda \ar[r] & \tilde{\pi}_{U*} \Lambda \ar[r] & \pi_{U*} \Lambda \ar[r]^{+1}&}.
\end{equation}

Since $\pi_U$ is a Grassmannian bundle, Lemma \ref{decomposition_grassmannian} implies that $\pi_{U*}\Lambda$ is formal\footnote{A sheaf is said to be formal if it is isomorphic to a shifted direct sum of its cohomology sheaves} and its cohomology sheaves are locally constant. Since $U$ is simply connected \cite[Expos\'e X, Corollaire 3.3]{SGA1}, they are in fact constant.~Let $\sM_{d-2,n-2}:=\oplus_i \underline{M^{i}_{d-2,n-2}}[-i]$\footnote{For any $\Lambda$-module $M$, by $\underline{M}$ we mean the constant local system on $\bbP \times \bbP$ with values in $M$.}, here $M_{d-2,n-2}^i:=H^0(U,R^{i}\pi_{U*}\Lambda)$. The restriction of $\sM_{d-2,n-2}$ to $U$ is isomorphic to $\pi_{U*} \Lambda$\footnote{The choice of $\sM_{d-2,n-2}$ is not unique in as much as the choice of the decomposition in Lemma \ref{decomposition_grassmannian}, but this non-uniqueness does not play a role in what follows.}. We also denote by $\sM_{d-1,n-1}:=\tilde{\pi}_*\Lambda$. We have exact triangles,

\begin{equation}\label{M1}
\xymatrix{j_{!} \pi_{Z!}\Lambda \ar[r] & \tilde{\pi}_* \Lambda \ar[r] &  \pi_* \Lambda \ar[r]^{+1}&},
\end{equation}

\begin{equation}\label{M2}
\xymatrix{ \sM_{d-1,n-1} \otimes j_{!}\Lambda \ar[r] & \sM_{d-1,n-1} \ar[r] & \sM_{d-1,n-1} \otimes \Delta_* \Lambda \ar[r]^-{+1} &}
\end{equation}

\noindent and 

\begin{equation}\label{M3}
\xymatrix{ \sM_{d-2,n-2} \otimes j_{!}\Lambda \ar[r] & \sM_{d-2,n-2} \ar[r] & \sM_{d-2,n-2} \otimes \Delta_* \Lambda \ar[r]^-{+1} &}
\end{equation}

\noindent in $\rD^b_c(\bbP \times \bbP)$. Now note that for any sheaf $K$ on $\bbP$ and any constant sheaf (i.e. the cohomology sheaves are constant) $L$ on $\bbP \times \bbP$,~the sheaf $s_{2*}(s_1^*K \otimes L)$ is also constant.~Thus combining triangles (\ref{M4})-(\ref{M3}) and Equation (\ref{Radon_inverse_Radon_2}) we get a functorial (in $K$) exact triangle in the localized category $\rD^b_c(\bbP)_{T}$,

\begin{equation}\label{M5}
\xymatrix{ \check{\cR} \circ \cR(K) \ar[r] & K\otimes \Delta^*\sM_{d-1,n-1}[\delta_+] \ar^-{\phi}[r] & K\otimes \Delta^*\sM_{d-2,n-2}[\delta_+]\ar[r]^-{+1} &}.
\end{equation}

\begin{claim}\label{perverse_Claim}

\begin{enumerate}[(a)]

\item  For any perverse sheaf $K$ on $\bbP$, there exists a natural isomorphism $\check{\cR}^{i} (\cR(K)) \simeq \check{\cR}^{i}(\cR^0(K))$ in $\cA(\bbP)$ (and hence in $D^b_c(\bbP,\Lambda)_T$).

\item For any perverse sheaf $K$ on $\bbP$, there exists functorial (in $K$) isomorphisms in $\cP(\bbP)$ (and hence in $\cA(\bbP)$) 

\begin{equation*}
    \cpH^i( K\otimes \Delta^*\sM_{d-1,n-1}[\delta_+]) \simeq K \otimes \Delta^*H^{i+\delta_+}(\sM_{d-1,n-1})
\end{equation*}

\noindent and 

\begin{equation*}
    \cpH^i( K\otimes \Delta^*\sM_{d-2,n-2}[\delta_+]) \simeq K \otimes \Delta^*H^{i+\delta_+}(\sM_{d-2,n-2}).
\end{equation*}

\item For $i=\delta-1,\delta$, the perverse sheaves $\cpH^i( K\otimes \Delta^*\sM_{d-2,n-2}[\delta_+])$ vanish. Also $ \cpH^{\delta-1}( K\otimes \Delta^*\sM_{d-1,n-1}[\delta_+])$ vanishes. Moreover when $n-d>1$, $\cpH^{\delta-2}( K\otimes \Delta^*\sM_{d-2,n-2}[\delta_+])$ is also zero.

\item For any perverse sheaf $K$ on $\bbP$, there exists a natural (in $K$) isomorphism in $\cP(\bbP)$ (and hence in $\cA(\bbP)$), $\cpH^{\delta}( K\otimes \Delta^*\sM_{d-1,n-1}[\delta_+])\simeq K(-d(n-d))$.


\end{enumerate}
\end{claim}

\begin{proof}
Claim (a) is an immediate consequence of Theorem \ref{image_Grassmannian_Brylynski_Radon}. Claims (b) follows from the formality of $\sM_{d-1,n-1}$ and $\sM_{d-2,n-2}$ and the fact that their cohomology sheaves are local systems.

For claim (c), using (b) it suffices to prove that $\Delta^*H^{i+\delta_+}(\sM_{d-2,n-2})$ vanishes for $\delta-2 \leq i \leq \delta$, and that $\Delta^*H^{\delta_++\delta-1}(\sM_{d-1,n-1})=0$. In either case note that the cohomology sheaves of $\sM_{d-2,n-2}$ and $\sM_{d-1,n-1}$ are constant local systems and hence by their definitions it suffices to show that $R^{i+\delta_+}\pi_{U*}\Lambda$ for $\delta-2 \leq i \leq \delta$ and $R^{\delta_++\delta-1}\tilde{\pi}_*\Lambda$ vanish. But these follow immediately from the fact that $\pi_U$ is a $\bG(d-2,n-2)$ bundle\footnote{We require $n-d>1$, to ensure that $\text{dim}(\bG(d-2,n-2))<d(n-d)-1$.} and that $\tilde{\pi}$ is a $\bG(d-1,n-1)$ bundle.

For claim (d) arguing as above we conclude that $\cpH^{\delta}( K\otimes \Delta^*\sM_{d-1,n-1}[\delta_+])\simeq K \otimes \Delta^*R^{2d(n-d)}\tilde{\pi}_*\Lambda \simeq K(-d(n-d))$.

\end{proof}
Combining claims (a)-(d) above shows that there exists a natural isomorphism $$\check{\cR}^{\delta}(d(n-d))\circ \cR^0(K) \simeq K$$ in $\cA(\bbP)$, and therefore complete the Proof of Theorem \ref{thm:mainthm} (2). It is also easy to see this map is the co-unit of the adjunction in Corollary \ref{adjunct_triple_perverse}.~Finally, combining Lemma \ref{behaviour_weights_duality} and Corollary \ref{adjunct_triple_perverse} we obtain the following.

\begin{cor}\label{unit_isomorphism}
The unit of the adjunction $K \to \check{\cR}^{-\delta}(d)\circ \cR^0(K)$ is an isomorphism in $\cA(\bbP)$.
\end{cor}

We also have the following corollary of the method of the proof.

\begin{cor}\label{fully_faithfullness}
We have $\Ext^{i}_{\cA(\bbP)}(K_1,K_2) \simeq \Ext^{i}_{\cA(Y)}(\cR^0(K_1),\cR^0(K_2))$ for $i=0,1$.
\end{cor}

\begin{proof}
The isomorphism for $i=0$ is an immediate consequence of Theorem \ref{thm:mainthm}, (2) and the adjunction between $\check{\cR}^{\delta}(d)$ and $\cR^0$ (Corollary \ref{adjunct_triple_perverse}). We may now assume that $n-d>1$, else the result follows from the fact that $\cR^0$ induces an equivalence between $\cA(\bbP)$ and $\cA(\check{\bbP})$ from from Theorem \ref{thm:mainthm}, (1) and (2).

The triangle (\ref{M5}) and Claim \ref{perverse_Claim}, (b) and (c) above imply that for $K \in \cA(\bbP)$,

\begin{equation*}
    \cpH^{-1}_T(\check{\cR}[\delta] \circ \cR K(d(n-d))) \simeq 0.
\end{equation*}

Since $\check{\cR}[\delta](d(n-d))$ is right $t$-exact and $\cR$ is exact, this implies that

\begin{equation}\label{FF3}
^p_T\tau_{\geq -1} \check{\cR}[\delta] \circ \cR K(d(n-d)) \simeq \check{\cR}^{\delta}(d(n-d))\circ \cR^0(K),
\end{equation}

which by Theorem \ref{thm:mainthm}, (2) is isomorphic to $K$ under the co-unit of adjunction.

We also have

\begin{equation}\label{FF1}
\Ext^{1}_{\cA(Y)}(\cR^0(K_1),\cR^0(K_2)) \simeq \text{Hom}_{\rD^b_c(\bbP)_{T}}(\check{\cR}[\delta] \circ \cR K_1(d),K_2[1])
\end{equation}

and

\begin{equation*}\label{FF2}
 \text{Hom}_{\rD^b_c(\bbP)_{T}}(\check{\cR}[\delta] \circ \cR K_1(d),K_2[1]) \simeq \Hom_{\rD^b_c(\bbP)_{T}}(^p_T\tau_{\geq -1}\check{\cR}[\delta] \circ \cR K_1(d(n-d)),K_2[1]).
\end{equation*}

The first equality being adjunction and the second since $K_1$ and $K_2$ are perverse, $\check{\cR}[\delta](d(n-d))$ is right $t$-exact and $\cR$ is exact. Combining this with (\ref{FF3}) gives the necessary equality.

\end{proof}

\subsection{Proof of Theorem \ref{thm:mainthm}, (3)}

\begin{proof}
Thanks to \ref{thm:mainthm}, (2) and Corollary \ref{fully_faithfullness}, it suffices to show that the simple objects in $\cA(Y)_{\text{Rad}}$ are in the image of $\cR^0$. This follows from Proposition \ref{key_prop_2}.
\end{proof}

Example \ref{failure_descent} naturally leads to the following question which we have been unable to answer:
\begin{question}
Does there exist a perverse sheaf on $Y$ with singular support inside $p_{2\circ}p_1^{\circ}\rT^*\bbP$ whose image is not in $\cA(Y)_{\text{Rad}}$, and hence the perverse sheaf is not in the image of the Radon transform?
\end{question}

Note that the answer to the above question is negative in characteristic $0$ (see Section \ref{comparison_previous}, (2)) or when $d=n-1$. 




{
$$
$$

\bibliographystyle{plain}
\bibliography{SSuppTransforms.bib}

\end{document}